\documentclass{article}
\usepackage[utf8]{inputenc}
\usepackage{amsmath,amssymb,amsthm}
\usepackage{graphicx}
\usepackage{graphics}
\usepackage{float}
\usepackage{epstopdf}
\newcommand{\yd}{{y_\delta}}

\newcommand{\xad}{{x_{\alpha,\delta}}}

\newcommand{\xa}{x_{\alpha}}
\newcommand{\xade}{{x_{\alpha,\delta,\eta}}}
\newcommand{\xades}{{x_{\alpha_*,\delta,\eta}}}
\newcommand{\as}{{\alpha_*}}

\newcommand{\xd}{x^\dagger}
\newcommand{\Ad}{A_\eta} 
\newcommand{\DA}{\Delta A}
\newcommand{\psim}{\bar{\psi}}

\newcommand{\ei}{e_{\text{per}}}
\newcommand{\bei}{\bar{e}_{\text{per}}}
\newcommand{\eii}{e_{\text{rel}}}
\newcommand{\eopt}{e_{\text{opt}}}
\newcommand{\val}{\theta}

\newtheorem{proposition}{Proposition}
\newtheorem{lemma}{Lemma}
\newtheorem{theorem}{Theorem}

\newtheorem*{remark}{Remark}

\newcommand{\Cpsima}{D}
\newcommand{\remove}[1]{}

\DeclareMathOperator*{\argmin}{argmin}

\usepackage[symbol]{footmisc}

\providecommand{\keywords}[1]{\textbf{Keywords:} #1}

\title{Semi-Heuristic Parameter Choice Rules for Tikhonov Regularisation with Operator Perturbations}
\author{Uno H\"amarik\thanks{Institute of Mathematics and Statistics, University of Tartu, J. Liivi 2, 50409, Tartu, Estonia (uno.hamarik@ut.ee \& urve.kangro@ut.ee)}, Urve Kangro\footnotemark[1], Stefan Kindermann\thanks{Industrial Mathematics Institute, Johannes Kepler University Linz, Altenbergerstra{\ss}e 69, 4040, Linz, Austria (kindermann@indmath.uni-linz.ac.at \& kemal.raik@indmath.uni-linz.ac.at)}, Kemal Raik\footnotemark[2]}

\begin{document}
\maketitle
\begin{abstract}
We study the choice of the regularisation parameter for linear ill-posed problems in the presence of 
data noise and operator perturbations, for which a bound on the operator error is known but 
the data noise-level is unknown.  We introduce a new family of semi-heuristic parameter choice rules 
that can be used in the stated  scenario. We prove convergence of the new rules and 
provide numerical experiments that indicate an improvement compared to standard 
heuristic rules. 
\end{abstract}


\section{Introduction}
The framework of this study are linear ill-posed problems with 
noisy data and an operator perturbation. The basis is the following 
well-known abstract model equation  
\begin{equation}
A x = y, \label{illposedproblem}
\end{equation}
with $A\in\mathcal{L}(X,Y)$, a continuous
linear operator acting between two Hilbert spaces with non-closed range, 
which, for simplicity, is furthermore
assumed to be injective. The contents of this paper remain valid for $A$ not injective, however.
In the following we denote by $\xd$ the minimum-norm least squares solution 
of~\eqref{illposedproblem}.

We assume that both the data and the operator are perturbed, i.e., 
\[ \yd = y  + e, \qquad \|e\|  \leq \delta, \] 
where $e$ denotes data error and $\delta$ the noise level.
The model is further corrupted as a consequence of the operator error
\[ \Ad = A + \DA, \qquad \|\DA\|  \leq \eta,\]
where $\DA$ is a bounded operator perturbation with  magnitude bounded by $\eta$.
We refer the reader to \cite{RausHam,tautenhahn,Pereverzyev,iteration} for further discussion regarding ill-posed problems with operator perturbations. There, one may find discussion of the \textit{generalised discrepancy principle} which is an a-posteriori parameter choice rule requiring knowledge of both the data and operator noise levels.

The specific  situation that we consider here, which is 
often met in practical situations, is that 
 we suppose we have knowledge of the operator noise level, i.e.,  
we assume {\em  $\eta$ known}, but  we do {\em not know} the level 
of the data error, $\delta$. 
 
It is an obvious fact  that such problems require regularisation, and  
for this study, we employ  Tikhonov regularisation: 
\begin{align}\label{tikreg}
 \xade &= (\Ad^* \Ad + \alpha I)^{-1} \Ad^* \yd, 
 \end{align} 
 with a regularisation parameter $\alpha$ and only the mentioned 
 bound on the perturbed operator available \cite{tikhon,engl}. 
 For later reference, we furthermore define two auxiliary functions
 \begin{align*}
 \xad &= (A^* A + \alpha I)^{-1} A^* \yd,  \\
 \xa &=  (A^* A + \alpha I)^{-1} A^* y.
 \end{align*} 
The choice of the regularisation parameter (here $\alpha$) is an important 
and delicate issue for any regularisation method.  The overall aim is  to obtain  convergence 
of the computed solution $\xade$ to the exact solution when 
all error terms $\delta,\eta$ vanish:
\begin{equation}\label{convergence}
\|\xade -\xd\| \to 0 \qquad \text{ as } \delta \to 0, \eta \to 0. \end{equation}
To this end, one must select a rule for choosing the appropriate parameter $\alpha$. 
If $\delta$ was known, there are parameter choice rules that provide such a convergence and even rates of convergence. 

However, when $\delta$ is unknown, as assumed 
in this paper, the choice of $\alpha$ is less standard and has 
to be done by {\em heuristic} rules, i.e., $\alpha$ is selected 
only depending on the given noisy data $\yd$ without explicit
reference to $\delta$. 
The best-understood 
methods in this class are the minimisation-based ones, on which we
build our methods as well. 

The novelty of this paper is the use and analyis 
of {\em semi-heuristic parameter} 
choice rules, where an assumed known operator bound, $\eta$, 
is combined with the $\delta$-free heuristic rules. 
This paper is organised as follows: in Section~\ref{sec:2}, 
we introduce and motivate the use of semi-heuristic rules,
and  we provide a convergence analysis. 
In Section~\ref{sec:4}, we illustrate the theory by 
numerical results. Additionally, whilst the reader may be referred to \cite{LuPerTaut} for the performance of the quasioptimality rule in the presence of a noisy operator, the performance of other heuristic rules in this setting has yet to be investigated. We subsequently shed new light on this as a byproduct of our comparison with the semi-heuristic rules.


\section{Semi-heuristic parameter choice rules}\label{sec:2}
As explained in the introduction, heuristic rules are employed 
in the case of unknown noise level $\delta$ (without operator 
perturbations).  

Minimisation-based heuristic rules entail   
 minimising a functional $\psi(\alpha,A,\yd)$
with 
\[\psi:[0,\alpha_{\text{max}}]\times\mathcal{L}(X,Y)\times Y\to\mathbb{R}\cup\{\infty\}.\] 
The regularisation parameter is then selected as 
\begin{equation}\label{defheu} \as = \argmin_{\alpha\in[0,\alpha_{\text{max}}]} \psi(\alpha,A,y_\delta),
\end{equation}
which obviously does not depend on $\delta$. 

Our methods  use  the following classical examples of 
$\psi$-functionals (see, e.g., \cite{convergenceanalysis}):
\begin{itemize}
\item The heuristic discrepancy functional
\begin{equation}\label{HRfun} 
\psi_{HD}(\alpha,A,y_\delta):=\frac{\|A x_{\alpha,\delta,\eta}-y_\delta\|}{\sqrt{\alpha}}.
\end{equation}
\item The Hanke-Raus functional
\begin{equation}\label{MHRfun} 
\psi_{HR}(\alpha,A,y_\delta):=
\|(A A^\ast+\alpha I)^{-1/2}(A x_{\alpha,\delta}-y_\delta)\|.
\end{equation} 
\item The quasioptimality functional
\begin{equation}\label{QOfun}  
\psi_{QO}(\alpha,A,y_\delta):=\left\|\alpha\frac{d}{d\alpha}
x_{\alpha,\delta}\right\|.
\end{equation}
\end{itemize}
In the linear case (as in the present setting), we can write these in terms of filter functions $\Psi$:
\[  \psi(\alpha,A,\yd)   = \|\Psi(\alpha,A)\yd\|. \]
In particular, the following filter functions
\begin{align*}
  \Psi_{HD}(\alpha,A)& := \sqrt{\alpha} (A A^* + \alpha I)^{-1}, \\ 
 \Psi_{HR}(\alpha,A) &:= \alpha (A A^* + \alpha I)^{-3/2},  \\ 
 \Psi_{QO}(\alpha,A) &:= \alpha (A^*A + \alpha I)^{-2}A^*
 \end{align*} 
may be associated with the heuristic discrepancy, Hanke-Raus
and quasioptimality rules, respectively \cite{hankeraus,hampalmraus}.

Meanwhile, a convergence theory for such heuristic parameter choice 
rules has also been established. A central ingredient is that a 
noise condition has to be postulated in order for these methods to work. 
Such a condition, which links the operator, the data error and 
the solution, provides a deep understanding when such methods are successful. 
Essentially, a noise condition is satisfied if the data noise is sufficiently
irregular. More precisely, we assume that for the specific choice 
of $\psi$, there exists a constant $C_{nc}$ such that for all 
given noisy data $\yd$ and  exact data $y$, the following inequality 
is satisfied 
\[ \|\xad - \xa \| \leq C_{nc} \|\Psi(\alpha,A )(\yd-y)\|   \qquad \forall \alpha \in [0, \alpha_{\text{max}}], \]
or, equivalently, 
\begin{equation}\label{nc} 
\|(A^* A + \alpha I)^{-1}A^*(\yd -y)\| \leq C_{nc} \psi(\alpha,A,\yd -y)
\qquad \forall \alpha \in [0, \alpha_{\text{max}}]. 
\end{equation} 
See \cite{convergenceanalysis,kinderquasi} for a more detailed discussion which gives more explicit representations and justification for the noise conditions using spectral theory.

Such inequalities are satisfied for the above mentioned classical 
$\psi$-functionals for many realistic instances of \textquotedblleft data noise\textquotedblright, e.g., 
for white or coloured noise  \cite{kinderquasi,KiPePi}.

The fact that prohibits the direct use of a minimisation-based 
rule with, say, a functional 
of the form $\psi(\alpha,\Ad,y_\delta),$
in our case,  is that we are faced with  an additional operator error, which 
is usually not  random or irregular and hence it would be unrealistic 
to assume that for the operator perturbation an analogous inequality holds.  
The remedy is to employ  a {\em modified functional}, which uses the 
noisy operator $\Ad$, but is designed to emulate a functional for the 
unperturbed operator. Therefore, we propose to subtract from the 
classical $\psi$-functional a term which should behave 
approximately like $\psi(\alpha,A,y_\delta)-\psi(\alpha,\Ad,y_\delta).$

Thus, the semi-heuristic rule is of the following type:
firstly, the regularisation parameter $\alpha = \as$ is chosen similarly to 
\eqref{defheu} by a minimisation of 
\begin{equation}\label{defpsim}
\psim(\alpha,\Ad,\yd)  :=  \psi(\alpha,\Ad,\yd)  - R(\alpha,\Ad,\yd,\eta)   
\end{equation} 
with $\psi$ being one of the classical heuristic functionals above, 
\eqref{HRfun}--\eqref{QOfun},
and 
a functional $R$ (to be specified below) that compensates the operator error. 
Secondly, to guarantee a minimiser and convergence of the regularized solution, we restrict the minimisation to 
an interval $[\gamma,\alpha_{max}]$, where the {\em  lower bound} $\gamma$
is selected depending on $\eta$ (but not on $\delta$):
\begin{equation}\label{semiheu}
\alpha_* := \alpha(\eta,\yd) :=  \argmin_{\alpha \in [\gamma,\alpha_{\text{max}}]} \psim(\alpha,\Ad,\yd), 
\qquad \gamma = \gamma(\eta)>0. 
\end{equation}
In this way, 
we combine heuristic rules with an $\eta$-based choice. 
 
We propose and investigate  two classes of compensating functionals $R$ labelled as 
(SH1) and (SH2).
 \begin{align}
 &\psim(\alpha,\Ad,\yd)  =  \psi(\alpha,A_\eta,y_\delta)  - \Cpsima \eta \|\xade\|,& \qquad 
 & \text{(SH1)} 
 \label{psimx}
 \\
 &\psim(\alpha,\Ad,\yd)  =  \psi(\alpha,A_\eta,y_\delta)  - \Cpsima \frac{\eta}{\sqrt{\alpha}}. 
 & \qquad 
 & \text{(SH2)} 
 \label{psima} 
\end{align}
The constant $D$ should be chosen to obtain a scaling invariant functional. For instance, in the case of (SH1), we may choose $D\sim 1/\|A\|$ and for (SH2), as $D\sim \|y\|/\|A\|$. Note that the error estimate we derive is sharpest with the choice \eqref{psimx}, although the numerical results are comparable.




The main goal of our analysis is to show convergence \eqref{convergence}
for such  semi-heuristic parameter choice rules.

\subsection{Error estimates with operator noise}
In the following, we assume the presence 
of operator noise. The following auxiliary result
will be utilised extensively.
\begin{lemma}\label{lem:opest}
Let $\alpha\in[0,\alpha_{\text{max}}]$ 
and for $p \in \{0,\frac{1}{2},1\}$, define 
\[ B_{\eta,p}:= 
\begin{cases} (A^\ast_\eta A_\eta)^p & \text{if }
p \in \{0,1\}, \\
A^\ast_\eta  & \text{if } p = \frac{1}{2}, 
\end{cases} 
\qquad 
B_{p}:= 
\begin{cases} (A^\ast A)^p   & \text{if }
p \in \{0,1\}, \\
A^\ast  & \text{if } p = \frac{1}{2}. 
\end{cases} 
\]
Let $ \hat{B}_{\eta,p}$ and $ \hat{B}_{p}$ be the operators we get from $ B_{\eta,p}$ and $B_{p}$ by changing the roles of the operators $A_\eta \leftrightarrow A^\ast_\eta$ and $A \leftrightarrow A^\ast$, respectively.
Then for 
$p \in \{0,\frac{1}{2},1\}$ and
$q \in \{-1,-\frac{3}{2},-2\}$, 
there exist positive constants $C_{p,q}$
such that
\begin{equation}
    \left\|(A^\ast_\eta A_\eta+\alpha I)^{q}
    B_{\eta,p}-(A^\ast A+\alpha I)^{q} B_{p}\right\|
    \le C_{p,q}\frac{\eta}{\alpha^{\frac{1}{2}-p-q}}. \label{operatorerrorestimates}
\end{equation}
\begin{equation}    
\left\|(A_\eta A^\ast_\eta+\alpha I)^{q}
    \hat{B}_{\eta,p}-(A A^\ast+\alpha I)^{q} \hat{B}_{p}\right\|
    \le C_{p,q}\frac{\eta}{\alpha^{\frac{1}{2}-p-q}}. \label{operatorerrorestimates2}
\end{equation}
\end{lemma}
\begin{proof}
We prove \eqref{operatorerrorestimates}, this gives \eqref{operatorerrorestimates2} changing the roles of the operators $A_\eta \leftrightarrow A^\ast_\eta$ and $A \leftrightarrow A^\ast$.
We recall the elementary 
estimates 
\begin{equation}\label{est} \|(A^\ast A +\alpha I)^{-1}\|
\leq \frac{1}{\alpha},
\quad 
\|(A^\ast A +\alpha I)^{-1}A^\ast\|
\leq \frac{1}{2 \sqrt{\alpha}},
\quad 
\|(A^\ast A +\alpha I)^{-1}A^\ast A\|
\leq 1,
\end{equation}
which also hold with $A$ and $A^\ast$ replaced by 
$A_\eta$ and $A_\eta^\ast$, respectively. 
For $p \in \{0,1\}$, it follows from 
some algebraic manipulations,
the fact that 
$B_p,B_{\eta.p}$ commute with the 
inverses below, 
and the previous estimates
that
\begin{align*}
&(A_\eta^\ast A_\eta +\alpha I)^{-1}B_{\eta,p}
- 
B_p (A^\ast A +\alpha I)^{-1}\\
&= 
(A_\eta^\ast A_\eta +\alpha I)^{-1} \left[
B_{\eta,p} (A^\ast A  + \alpha I) 
- (A_\eta ^\ast A_\eta +\alpha I) B_p
\right]
(A^\ast A +\alpha I)^{-1}\\
& = 
(A_\eta^\ast A_\eta +\alpha I)^{-1}\left[ B_{\eta,p} A^\ast A - 
A_\eta^\ast A_\eta B_p \right] 
(A^\ast A +\alpha I)^{-1} \\
& \qquad +
\alpha (A_\eta^\ast A_\eta +\alpha I)^{-1}
\left[ B_{\eta,p} - B_p  \right]  (A^\ast A +\alpha I)^{-1}. 
\end{align*}
In the case $p = 0$ and 
$B_{\eta,0}= B_0 = I$, we find
\[ B_{\eta,0} A^\ast A - 
A_\eta^\ast A_\eta B_0
= (A^\ast-A_\eta^\ast)A + 
A_\eta^\ast(A -A_\eta), \]
which, using \eqref{est}, gives $C_{0,-1}= 1$.
Similarly, we can 
prove that 
$C_{1,-1}= 1$. 
For the case $p = \frac{1}{2}$, 
if $B_{\eta,p} = A_\eta^\ast$ and $B_{p} = A^\ast$, 
we obtain 
$C_{\frac{1}{2},-1}= \frac{5}{4}$
with minor modifications noting that 
$(A^\ast A +\alpha I)^{-1}A^\ast
= A^\ast (AA^\ast +\alpha I)^{-1}$. 
The other cases of $q$ follow
in a similar manner by 
\begin{align*}
&(A_\eta^\ast A_\eta +\alpha I)^{q}B_{\eta,p}
- 
B_p (A^\ast A +\alpha I)^{q}\\
&= 
(A_\eta^\ast A_\eta +\alpha I)^{q+1} \left[
(A_\eta^\ast A_\eta +\alpha I)^{-1}
B_{\eta,p} -
B_p (A^\ast A +\alpha I)^{-1}
\right]  \\
&+
\left[
(A_\eta^\ast A_\eta +\alpha I)^{q+1}
 -
 (A^\ast A +\alpha I)^{q+1}\right] B_p(A^\ast A +\alpha I)^{-1},
\end{align*}
and by  using \eqref{est} 
and the result for $q = -1$. 
For $q = -\frac{3}{2}$, we employ an additional identity 
from semigroup operator calculus \cite{krasnosleski}, 
\begin{align*}
&(A_\eta^\ast A_\eta +\alpha I)^{-\frac{1}{2}}-
(A^\ast A +\alpha I)^{-\frac{1}{2}}\\
& \quad =
\frac{\sin(\frac{\pi}{2})}{\pi}\int_0^\infty
t^{-\frac{1}{2}} \left[
(A_\eta^\ast A_\eta +(\alpha+t) I)^{-1} - 
(A^\ast A +
(\alpha+t) I)^{-1}\right] \,\mathrm{d}t,
\end{align*} 
which leads to 
\begin{align*}
\|(A_\eta^\ast A_\eta +\alpha I)^{-\frac{1}{2}}-
(A^\ast A +\alpha I)^{-\frac{1}{2}}\| &\leq 
 \frac{C_{0,-1} \eta}{\pi} 
\int_0^\infty \frac{1}{\sqrt{t}(\alpha +t)^\frac{3}{2}} \,\mathrm{d}t\\
& \leq  \frac{2 C_{0,-1} }{\pi}
\frac{\eta}{\alpha},
\end{align*}
thereby finishing the proof. 
\end{proof}
As a consequence of the above lemma, we obtain some useful bounds.
\begin{lemma}\label{helplemma} 
For any of the parameter choice functionals 
$\psi \in \{\psi_{HD},\psi_{HR}, \psi_{QO} \}$
(see \eqref{HRfun}--\eqref{QOfun}), 
any  $\alpha \in [0,\alpha_{\text{max}}]$ 
we have 
\begin{equation}\label{psi_eta}
\psi(\alpha,A_\eta,A_\eta \xd) 
\leq  C_{p,q}\frac{\eta \|\xd\|}{\sqrt{\alpha}} + 
\psi(\alpha,A,A \xd), 
\end{equation}
\begin{equation}\label{psiupper}
\psi(\alpha,A_\eta,y_\delta)
\leq \frac{\delta}{\sqrt{\alpha}} + 
(1+C_{p,q})\frac{\eta \|\xd\|}{\sqrt{\alpha}}
+ \psi(\alpha,A,A \xd),
\end{equation}
with the constants $C_{p,q}$ from 
Lemma~\ref{lem:opest}: 
$p = \frac{1}{2},$ $q= -1$
for the heuristic discrepancy, 
 $p = \frac{1}{2},$ $q= -\frac{3}{2}$
for the  Hanke-Raus, and
 $p = 1,$ $q= -2$ for the 
quasioptimality functionals, respectively.
\end{lemma}
\begin{proof}

%

The inequality \eqref{psi_eta} follows from \eqref{operatorerrorestimates} and \eqref{operatorerrorestimates2}, the inequality \eqref{psiupper} from \eqref{psi_eta} and from the inequalities
\begin{align*}  \psi(\alpha,A_\eta,y_\delta)
&\leq \psi(\alpha,A_\eta,y_\delta-y)
+\psi(\alpha,A_\eta,(A-A_\eta)\xd) +
\psi(\alpha,A_\eta,A_\eta \xd) 
\\
&\leq \frac{\delta}{\sqrt{\alpha}} 
+ \frac{\eta \|\xd\|} {\sqrt{\alpha}} 
+ \psi(\alpha,A_\eta,A_\eta \xd).
\end{align*}  
%
\end{proof}
We remark that the term 
$\psi(\alpha,A,A \xd)$ converges to $0$ as 
$\alpha \to 0$; see, e.g.,~\cite{kinderquasi}. 
Furthermore, if $\xd$ additionally satisfies a source 
condition \cite{engl}, then the expression can be bounded 
by a convergence rate of order $\alpha$ (with some exponent depending on the source condition) that 
agrees with the standard 
rate for the approximation error 
$\|\xa -\xd\|$.

\subsection{Convergence}
Suppose that $\alpha_*$ is the selected parameter by the proposed parameter choice rules with the operator noise \eqref{semiheu}. 
In the following lemma, we show that for such a choice of parameter, it follows that $\as \to 0$ if all noise (with respect to both the data and the operator) vanishes:
\begin{lemma}\label{lemma3}
Let $\alpha_\ast$ be selected as above, i.e., \eqref{semiheu}, 
with $\psim$ as in \eqref{psimx} or \eqref{psima} and 
$\psi \in \{\psi_{HD},\psi_{HR},\psi_{QO}\}$.
Suppose there exist positive constants (not necessarily equal which we denote universally by $C$) such that $\|y_\delta\|\ge C$ for $\psi\in\{\psi_{HD},\psi_{HR}\}$ and $\|A_\eta^\ast y_\delta\|\ge C$ for $\psi=\psi_{QO}$.

If $\gamma = \gamma(\eta)$ is chosen such that 
$\frac{\eta}{\sqrt{\gamma}}\to 0$ as $\eta \to 0$
then 
\[\alpha_\ast\to 0\] 
as $\delta,\eta\to 0$.
\end{lemma}
\begin{proof}
At first, we show some lower bounds 
for the parameter choice functionals. 
 If $\|\yd\|\geq c_0$ and $\|A_{\eta}^\ast\yd\| \geq c_0$ then there exist constants such that
\begin{equation}\label{lower}
\psi(\alpha,A_\eta,\yd) \geq 
\begin{cases} 
C \sqrt{\alpha} & \text{ if } 
\psi =  \psi_{HD}, \\
C \alpha & \text{ if } 
\psi \in  \{ \psi_{HR}, \psi_{QO}\}.  
\end{cases}
\end{equation}
To see this, we get from the relation (here $s \geq 0$ arbitrary)
\[ \| \yd \|  = 
\|(A_\eta A_{\eta}^\ast + \alpha I)^{s} (A_\eta A_{\eta}^\ast + \alpha I)^{-s} \yd\| \leq 
\|(A_\eta A_{\eta}^\ast + \alpha I)^{s}\| \|(A_\eta A_{\eta}^\ast + \alpha I)^{-s} \yd\|   \]
the inequality 
\[\|(A_\eta A_\eta^\ast + \alpha I)^{-s} \yd\|  \geq \frac{\|\yd\|}{\|(A_\eta A_\eta^\ast + \alpha I)^{s}\|} \geq 
 \frac{c_0}{c_1^{s}}, \]
with $c_1 \geq \|A_\eta A_\eta^\ast + \alpha I\|$. This gives \eqref{lower} for $\psi_{HD}$ ($s=1$) and $\psi_{HR}$ ($s=\frac{3}{2}$). The estimate \eqref{lower} for $\psi_{QO}$ follows analogously: 
\[ \psi_{QO} = \alpha 
\|(A_{\eta}^\ast A_\eta  + \alpha I)^{-2} A_{\eta}^\ast \yd\|  \geq \alpha \frac{\|A_{\eta}^\ast\yd\|}{\|(\|A_{\eta}^\ast A_\eta + \alpha I)^{2}\|} \geq 
\alpha  \frac{c_0}{c_1^{2}}. \]
 
By the standard error estimate 
\[ \|\xades\| \leq \frac{\|\yd\|}{\sqrt{\alpha_\ast}}, \]
we find, for the case in which the compensating functional is chosen as in \eqref{psimx} using \eqref{lower} and \eqref{psiupper}
with $t\in\{1/2,1\}$ suited to $\psi$ according to \eqref{lower}, 
\begin{align}
&C \as^t - \Cpsima\eta 
\frac{\|\yd\|}{\sqrt{\as}} \leq 
\psim(\as,A_\eta,y_\delta) = 
\inf_{\alpha \in [\gamma,\alpha_{\text{max}}]}
\psim(\alpha,A_\eta,y_\delta)  \nonumber \\
&\leq \inf_{\alpha \in [\gamma,\alpha_{\text{max}}]}
\psi(\alpha,A_\eta,y_\delta) \leq
\inf_{\alpha \in [\gamma,\alpha_{\text{max}}]}
\left\{
\frac{\delta}{\sqrt{\alpha}} + 
(1+C_{p,q})\frac{\eta \|\xd\|}{\sqrt{\alpha}}
+ \psi(\alpha,A,A \xd) \right\} 
\nonumber \\
&\leq 
\inf_{\alpha \in [\gamma,\alpha_{\text{max}}]}
\left\{
\frac{\delta}{\sqrt{\alpha}} + 
+ \psi(\alpha,A,A \xd) \right\} 
+(1+C_{p,q})\frac{\eta \|\xd\|}{\sqrt{\gamma}}.  \label{upperpsim}
\end{align} Hence, 
\begin{align*}
&C \as^t \leq 
\inf_{\alpha \in [\gamma,\alpha_{\text{max}}]}
\left\{
\frac{\delta}{\sqrt{\alpha}} 
+ \psi(\alpha,A,A \xd) \right\} +
(C +\Cpsima )\frac{\eta}{\sqrt{\gamma}}\, .
\end{align*}
It is not difficult to verify the same estimate analogously for the
case in which the compensating functional is chosen according to \eqref{psima}.  

Inserting  the (nonoptimal) choice $\alpha = \delta + \gamma$ in the  infimum,
we obtain an upper bound that 
tends to $0$ 
as $\delta,\gamma \to 0$. By the hypothesis,
the last two terms vanish, thereby proving the desired result. 
\end{proof}

\begin{remark}
If $\alpha_\ast$ is the minimizer of $\psi(\alpha,A_\eta,y_\delta)$, then this functional is the same as
 \eqref{psimx} and/or \eqref{psima} with $D=0$ and one obtains the 
same result as above; namely, that $\alpha_\ast\to 0$ as $\delta,\eta\to 0$ provided that the
conditions in the lemma are fulfilled.
\end{remark}

%
Now, we can establish an estimate from above for the total error which is derived courtesy of a 
lower estimate of the parameter choice functional with the data error. Note that, due to Bakushinskii's veto,
this estimate cannot be derived without restricting the set of permissible noise \cite{veto}, e.g., 
by a noise condition. At first we study bounds for the functional in \eqref{semiheu}.

\begin{proposition}\label{prop2}
Let $\as$ be selected according to 
\eqref{semiheu} with $\psim$ as in \eqref{psimx}.
Suppose that for the noisy data $\yd$,
the noise condition \eqref{nc} 
is satisfied. Then, for $\eta$ sufficiently small, we get the following error estimate for all $\psi\in\{\psi_{HD},\psi_{HR},\psi_{QO}\}$:
\begin{equation}
\begin{aligned}
&
\|\xades -\xd\|     
\\&
\le (1-\Cpsima\eta C_{nc})^{-1} \bigg[C\frac{\eta\delta}{\alpha_\ast}+C_{nc}\inf_{\alpha\in[\gamma,\alpha_{\text{max}}]}\psim(\alpha,A_\eta,y_\delta)+\Cpsima C_{nc}\eta\|x^\dagger\| \\
&
\qquad \qquad  +C\frac{\eta}{\sqrt{\alpha_\ast}}\|x^\dagger\|+C\psi(\alpha_\ast,A,A\xd) +\as \|(A^* A + \as I)^{-1} \xd \| 
\bigg].
\end{aligned} \label{totalerror}
\end{equation}
\end{proposition}
\begin{proof}
We begin by estimating the terms: 
\begin{align*} 
&\|\xades -\xd\|  = 
\|(\Ad^* \Ad + \as I)^{-1}\Ad^* \yd   -
(\Ad^* \Ad + \as I)^{-1} (\Ad^* \Ad + \as I)
\xd \|  \\
& \leq 
\|(\Ad^* \Ad + \as I)^{-1}\left[ \Ad^* \yd  - \Ad^* \Ad \xd - \as \xd \right] \|  \\
&  \leq \|(\Ad^* \Ad + \as I)^{-1}\left[\Ad^*(\yd -y) + \Ad^*(A -\Ad)\xd\right]\| + 
\as \|(\Ad^* \Ad + \as I)^{-1} \xd \|  \\
& \leq  \|(\Ad^* \Ad + \as I)^{-1}\Ad^*(\yd -y)\| + 
\frac{\eta}{2\sqrt{\as}} \|\xd\|   + \as \|(\Ad^* \Ad + \as I)^{-1} \xd \|.  
\end{align*}
By \eqref{operatorerrorestimates}, the last term can be bounded by 
\begin{align*}   \as \|(\Ad^* \Ad + \as I)^{-1} \xd \|  &\leq 
 \as \|\left[(\Ad^* \Ad + \as I)^{-1} -  (A^* A + \as I)^{-1} \right]\xd  \| \\
 & \qquad \qquad \qquad \qquad 
+ \as \|(A^* A + \as I)^{-1} \xd \|  \\
& \leq C_{0,-1} \frac{\eta \|\xd\|}{\sqrt{\alpha_\ast}} +
 \as \|(A^* A + \as I)^{-1} \xd \|. 
 \end{align*}
This leaves the remaining term:
\begin{align*}
&\|(\Ad^* \Ad + \as I)^{-1}\Ad^*(\yd -y)\|\\
&\leq 
\|\left[(\Ad^* \Ad + \as I)^{-1}\Ad^* - (A^* A + \as I)^{-1}A^*\right] (\yd -y)\|  \\
& \qquad \qquad \qquad + 
\|(A^* A + \as I)^{-1}A^* (\yd -y)\|  \\
& \leq \frac{5 \eta \delta }{4 \as} + 
C_{nc} \psi(\as,A,\yd-y). 
\end{align*}
Combining the noise condition with the operator error estimates \eqref{operatorerrorestimates}, \eqref{operatorerrorestimates2} 
we obtain
\begin{align*}
&\|(\Ad^* \Ad + \as I)^{-1}\Ad^*(\yd -y)\|
 \leq 
  \frac{5 \eta \delta }{4 \as} + 
 C_{nc} \psi(\as,A_\eta,\yd-y) + 
 C_{nc} C_ {p,q} \frac{\delta \eta}{\as} \\
 & \leq 
   \frac{(5 +  C_{nc} C_ {p,q} ) \eta \delta }{4 \as}
 +  C_{nc} \psi(\as,A_\eta,\yd) + 
  C_{nc}  \psi(\as,A_\eta,y) \\
  &\leq 
     \frac{(5 +  C_{nc} C_ {p,q} ) \eta \delta }{4 \as} + 
 C_{nc}\psim(\as,A_\eta,\yd) + 
  \Cpsima C_{nc} \eta \|\xades\| +  C_{nc}
   \psi(\as,A_\eta,Ax^\dagger) \\
&  \leq 
     \frac{C \eta \delta }{ \as} + 
 C_{nc} \inf_{\alpha \in [\gamma,\alpha_{\text{max}}]} \psim(\alpha,A_\eta,\yd) + 
  \Cpsima C_{nc} \eta \|\xades -\xd\| +
    \Cpsima C_{nc} \eta \|\xd\| \\
    & \qquad + C_{nc}  \psi(\as,A_\eta,(A_\eta+A-A_\eta)x^\dagger) \\
& \leq 
     \frac{C \eta \delta }{ \as} + 
 C_{nc} \inf_{\alpha \in [\gamma,\alpha_{\text{max}}]} \psim(\alpha,A_\eta,\yd) + 
  \Cpsima C_{nc} \eta \|\xades -\xd\| +
    \Cpsima C_{nc} \eta \|\xd\|\\
    & \qquad \qquad +
    C_{nc}\psi(\alpha_\ast,A_\eta,A_\eta x^\dagger)+C_{nc}\psi(\alpha_\ast,A_\eta,(A-A_\eta)x^\dagger).
 \end{align*}
 The last terms can be bounded using standard error estimates by 
 \begin{align*} 
 \psi(\alpha_\ast,A_\eta,(A-A_\eta)x^\dagger) \leq \frac{1}{\sqrt{\alpha}} 
 \|(A-A_\eta)x^\dagger\| = \frac{\eta \|\xd\|}{\sqrt{\alpha}}, 
\end{align*}  
 while for the other term we employ
 \eqref{operatorerrorestimates} and \eqref{psi_eta}
 \begin{align*}
     \psi(\alpha_\ast,A_\eta,A_\eta x^\dagger) \leq 
     C_{p,q} \frac{\eta \|\xd\|}{\sqrt{\alpha}}  + 
          \psi(\alpha_\ast,A,A x^\dagger). 
 \end{align*}
Hence, for all $\psi\in\{\psi_{HD},\psi_{HR},\psi_{QO}\}$, we obtain 
\begin{align*}
&(1-\Cpsima \eta C_{nc}) \|\xades -\xd\|     \\
&\le\frac{C\eta\delta}{\alpha_\ast}+C_{nc}\inf_{\alpha\in[\gamma,\alpha_{\text{max}}]}\psim(\alpha,A_\eta,y_\delta)+\Cpsima C_{nc}\eta\|x^\dagger\| 
\\
&
+C\frac{\eta}{\sqrt{\alpha_\ast}}\|x^\dagger\|+C\psi(\alpha_\ast,A,A \xd)
+\as \|(A^* A + \as I)^{-1} \xd \|. 
 \end{align*}
\end{proof}
%

The proof is easily adapted to obtain a  
similar  proposition for the alternative choice of compensating functional as in 
\eqref{psima}:
\begin{proposition}
Let the assumptions of the Proposition~\ref{prop2} hold. 
Let $\as$ be selected according to 
\eqref{semiheu} with $\psim$ as in \eqref{psima}. Then for $\eta$ sufficiently small, we get
\begin{equation}
\begin{aligned}
&\|\xades -\xd\|     \\
&\le  C\frac{\eta\delta}{\alpha_\ast}+C_{nc}\inf_{\alpha\in[\gamma,\alpha_{\text{max}}]}\psim(\alpha,A_\eta,y_\delta)+C_{nc}C \frac{\eta}{\sqrt{\as}} \\
&
\qquad \qquad  +C\frac{\eta}{\sqrt{\alpha_\ast}}\|x^\dagger\|+C\psi(\alpha_\ast,A,A\xd) 
+\as \|(A^* A + \as I)^{-1} \xd \| .
\end{aligned} \label{totalerror1}
\end{equation}
\end{proposition}
Note that the setting $D=0$ in the previous propositions yields upper bounds for the 
total errors in the case of employing the unmodified heuristic rules.


Thus, with the estimate above, we can prove the desired convergence theorem providing that certain conditions are satisfied:
\begin{theorem}\label{theorem1}
	Let $\alpha_\ast$ be selected as in \eqref{semiheu}. Suppose that the noise condition \eqref{nc} and the 
	conditions of Lemma~\ref{lemma3} are satisfied and furthermore suppose 
	that $\gamma\in[0,\alpha_{\text{max}}]$ satisfies 
\[
 \frac{\eta}{\gamma} \leq C  
\qquad \mbox{ as } \eta \to 0, \] 
where $C$ is a  constant. 
	Then
	\[ \|x_{\alpha_\ast,\delta,\eta}-x^\dagger\|\to 0 \]
	as $\delta,\eta\to 0$.
\end{theorem}

\begin{proof}
Since we have that $\as \geq \gamma$, the conditions in the theorem 
imply that 
$\tfrac{\eta \delta}{\gamma} \to 0,$
$\tfrac{\eta}{\sqrt{\gamma}} \to 0$.
The terms with $\psi(\alpha_\ast,A,A\xd)$ and  
$\as \|(A^* A + \as I)^{-1} \xd \| $ vanish by standard arguments because 
$\as \to 0$ according to Lemma~\ref{lemma3}. 
Finally, 
$\inf_{\alpha\in[\gamma,\alpha_{\text{max}}]}\psim(\alpha,A_\eta,y_\delta)$
tends to $0$ because of \eqref{upperpsim} and we may take 
an appropriate choice for $\alpha$ in the infimum.
\end{proof}

\begin{remark}
Note that one might use more general functionals than those in 
\eqref{psimx} and \eqref{psima} by replacing $\eta$ with $\eta^s$, $s \in (0,1)$. 
Still, in this case, similar convergence results are valid with a slightly 
adapted choice of $\gamma$ (depending on $s$).  
However, we observed through some numerical experimentation that $s=1$ appeared to be  
a natural choice, which is fully in line with  our motivation that the compensating term 
should represent the error in $\psi$ due to operator perturbations.

We further remark that the unmodified heuristic choice (i.e., with $D=0$), 
stipulating the same condition as in the 
previous theorem, also yields convergence as the errors  tend to zero.
However, as will be observed in Section~\ref{sec:4}, the modified rules represent a 
substantial improvement. 
\end{remark}

\section{Numerical experiments}\label{sec:4}
In this section, we test the numerical performance of the various modified functionals, 
$\psim$, on a series of test problems.  We provide two types of experiments: one with 
random operator noise and the other with a smooth operator perturbation.
Note that heuristic rules can fail in the case of smooth 
errors that do not satisfy a noise condition. Thus, a smooth operator perturbation is the 
most critical case for heuristic rules, and, as we will observe, the
semi-heuristic methods will prove to be more effective  in that case.

For each of the proposed heuristic rules, 
we compute the relative  error with respect to the selected regularisation parameter $\alpha_\ast$
\begin{equation*}
    \eii:=\frac{\|x_{\alpha_\ast,\delta,\eta}-x^\dagger\|}{\|x^\dagger\|},
\end{equation*}
and the error obtained by the 
theoretically optimal choice of $\alpha$
\begin{equation*}
    \eopt:=\frac{\|x_{\alpha_{\text{opt}},\delta,\eta}-x^\dagger\|}{\|x^\dagger\|}, 
    \qquad \alpha_{\text{opt}}:=\argmin_{\alpha}\|x_{\alpha,\delta,\eta}-x^\dagger\|.
\end{equation*}
Furthermore, we denote the ratio of these errors by 
\[\ei:=\frac{\|x_{\alpha_\ast,\delta,\eta}-x^\dagger\|}{\|x_{\alpha_{\text{opt}},\delta,\eta}-x^\dagger\|}.
\]

Note that in our simulations, we are afforded the luxury of knowing $\xd$, thereby allowing us to 
minimise the error and compute $\alpha_{\text{opt}}$ and $\eopt$.

For the standard heuristic rules, we search for $\alpha\in[\lambda_{\text{min}},\|A\|^2]$, 
where $\lambda_{\text{min}}$ is the minimum eigenvalue of the matrix $A^\ast A$. (However, if $\lambda_{\text{min}}$ is below $10^{-14}$, then we choose $\alpha_{\text{min}}=10^{-14}$ to avoid numerical instabilities). In some cases of large operator noise, the heuristic rules selected $\alpha_\ast=\alpha_{\text{max}}$; thus in this situation, we select the parameter corresponding to the smallest interior local minimum. For the semi-heuristic rules, however, we restrict our search to the interval $[\gamma,\alpha_{\text{max}}]$, where $\gamma=O(\eta)$ as above. Furthermore, in each experiment, we have scaled the operator and the exact solution so that $\|A\|=1$ and $\|x^\dagger\|=1$.

For numerical comparisons of standard heuristic rules in the absence of operator noise, we refer to \cite{BauLuk,hampalmraus,hapara11,HaHa,RaHa}.

\subsection{Gau{\ss}ian operator noise perturbation}
\paragraph{Tomography operator perturbed by Gau{\ss}ian operator.}
In this experiment, we use the \texttt{tomo} package from Hansen's 
Regularisation Tools~\cite{hansentool}
to define the finite-dimensional operator
(i.e., matrix) $A_\eta\in\mathbb{R}^{n\times n}$, where $A_\eta=A+C\Delta A$, with $A$
the tomography operator, which is a penetration of a two dimensional domain by rays in random directions. 
We use random Gau{\ss}ian distributed operator noise, i.e., 
$\Delta A\in\mathbb{R}^{n\times n}$ is a matrix 
with random entries. 
The data noise is defined as $\delta=C\|\epsilon\|$, where $\epsilon\in\mathbb{R}^n$ is a
Gau{\ss}ian distributed noise vector.

In the following configuration, we set $n=625$ and $\texttt{f}=1$, according to Hansen's Tools.

We provide a dot plot, namely Figure~\ref{tomdot1}, in which we compare the error $\eii$
according to the relative error function
for each parameter choice rule and for 100 different realisations of data errors 
and operator perturbations with values of $\delta$ and $\eta$ 
ranging from 1\% to 10\%. Each asterisk in the plot corresponds to the relative error, $\eii$, for a realisation of operator and data noise. 

Note that ``semi-heuristic rule 1'' and ``semi-heuristic 
rule 2'' in Figure~\ref{tom1} refer to the modified rules with   $\eta\|x_{\alpha,\delta,\eta}\|$,
cf.~\eqref{psimx}, 
and $\eta/\sqrt{\alpha}$, cf.~\eqref{psima}, as compensating functionals, respectively. Recall that the standard heuristic rules (in blue) correspond to the semi-heuristic rules with $D=0$ and search for a parameter $\alpha$ in the interval $[\lambda_{\text{min}},\|A\|^2]$. The last row in the plot is then the dot plot of the relative error for the optimal choice of $\alpha$, namely $e_{\text{opt}}$. In each row, the green circles represent the median of the respective relative errors over the 100 realisations. 

\begin{figure}
	\centering
	\includegraphics[scale=.8]{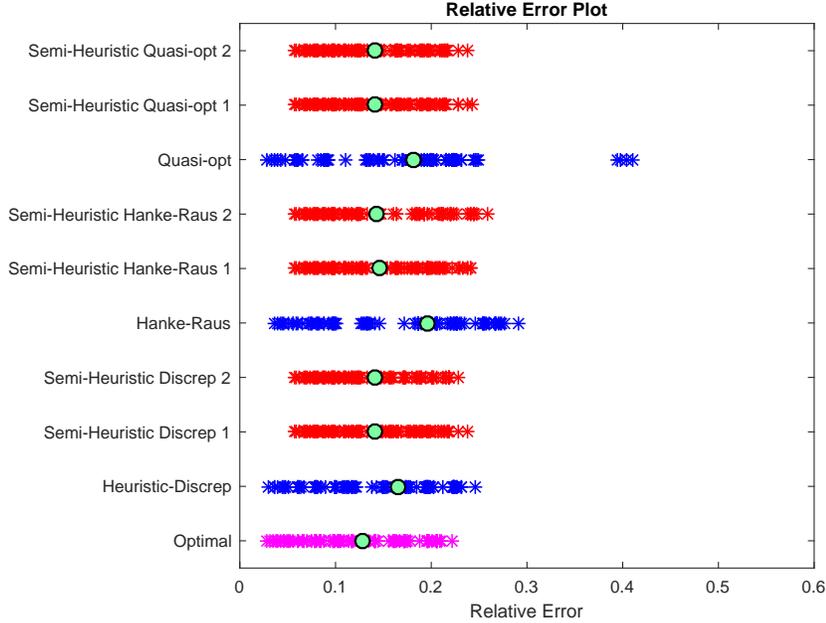}
	\caption{Tomography operator perturbed by random operator: $D=600$ for SH1, $D=0.05$ for SH2 and $\gamma=0.005\times\eta$.}
	\label{tomdot1}
\end{figure}

We see that the semi-heuristic rules present a noticeable improvement for all parameter choice rules, although the discrepancy in performance seems to be slightly more pronounced for the quasioptimality and Hanke-Raus rules.

We also compare
the difference between the values of $\ei$ with respect to the modified parameter choice rule
and its unmodified counterpart, respectively as a percentage.
For example, for any configuration of data and operator noise, we would compute 
the value 
\begin{equation}\label{defzeta}   \val(\delta,\eta) =  (\ei-\bei)\times 100,  \end{equation} 
where $\ei$ and $\bei$ 
denote the error ratio for the standard heuristic rule (i.e., $D = 0$ and $\alpha_{\text{min}}=\lambda_{\text{min}}$) 
and the modified rule \eqref{psimx} or  \eqref{psima}, 
respectively. This value is computed for several noise-levels $\delta$ and operator error levels 
$\eta$. Note that positive values indicate that the semi-heuristic rules outperform their heuristic counterparts and vice versa.

\begin{figure}
	\centering
	\includegraphics[scale=.8]{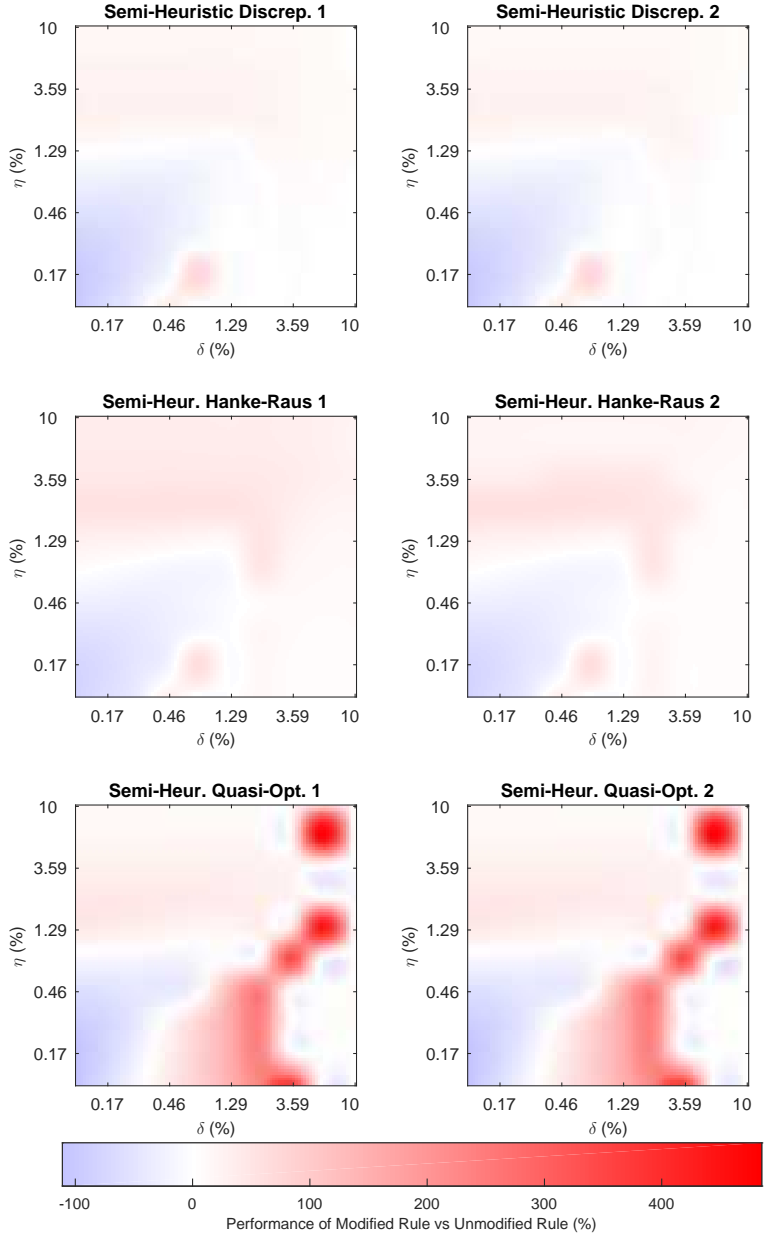}
	\caption{Tomography operator perturbed by random operator: set-up identical to Figure~\ref{tomdot1}. Red indicates that the semi-heuristic rules perform better than their standard heuristic counterparts and vice versa.}
	\label{tom1}
\end{figure}

The plots of Figure~\ref{tom1} indicate that the semi-heuristic rules do not necessarily offer improvements for small data and operator noise, but exhibit increased performance for larger noise of both aforementioned varieties. In particular, this is more pronounced for the quasioptimality rule where we may observe blotches of dark red which indicate significant improvement over the standard heuristic rule. 

The standard heuristic rules also performed reasonably well and a possible explanation could be the argumentation 
for the use of the compensating functional was based on the regularity of the operator noise and 
therefore it is probable that the irregularity of the operator noise in this scenario did not aid the
premise of using the modified rules.

\begin{figure}
	\centering
	
\end{figure}

\subsection{Smooth Operator Perturbation}
\paragraph{Fredholm integral operator perturbed by heat operator}
To simulate a deterministic, possibly smooth, operator perturbation, we first consider the 
Fredholm integral operator of the first kind perturbed by a heat operator, which we think is an instance
where the noise condition for $A_\eta$ might fail and where a semi-heuristic modification 
is highly advisable. 

For the implementation, 
we use the \texttt{baart} and \texttt{heat} packages on Hansen's Regularization Tools to 
define the finite dimensional operator $A_\eta\in\mathbb{R}^{n\times n}$, with $n=400$, where 
$A_\eta=A+C \Delta A$ is the superposition of the \texttt{baart} operator and scaled 
heat operator, respectively. More precisely, the \texttt{baart} operator is the discretisation of a Fredholm integral equation of the first kind with kernel $K_1:(s_1,t_1)\mapsto\exp(s_1\cos t_1)$, where $s_1\in[0,\pi/2]$, $t_1\in[0,\pi]$, and the heat operator is taken to be the Volterra integral operator with kernel $K_2:(s_2,t_2)\mapsto k(s_2-t_2)$, where
\[k(t):=\frac{t^{-\frac{3}{2}}}{2\sqrt{\pi}}\exp\left(-\frac{1}{4t}\right), \]
for $t\in[0,1]$. The exact solution is given by $y(s)=2\sin s/s$ and the data noise is defined as before. 

We proceed similarly as in the previous experiment.

\begin{figure}
	\centering
	\includegraphics[scale=.8]{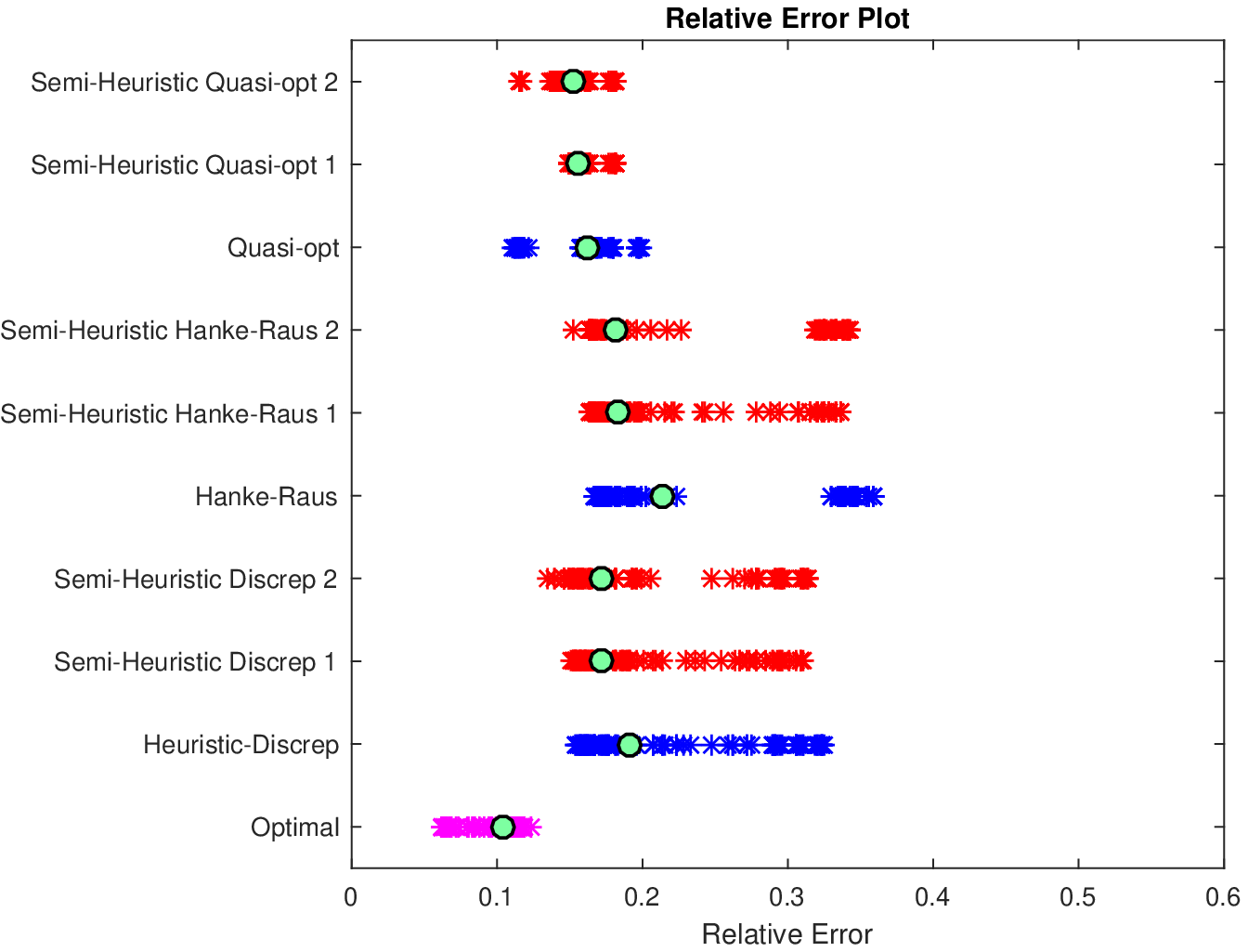}
	\caption{Fredholm operator of the first kind perturbed  by heat operator: $D=600$ for SH1, $D=0.12$ for SH2 and $\gamma=0.07\times\eta$.}
	\label{heatdot1}
\end{figure}

In Figure~\ref{heatdot1}, we observe that the best performing rule is in fact the semi-heuristic quasioptimality rule (SH1). The semi-heuristic variants of the Hanke-Raus and heuristic discrepancy rules are also improvements on the original rules, although this is slightly more pronounced for the semi-heuristic Hanke-Raus rules.

\begin{figure}
	\centering
	\includegraphics[scale=.8]{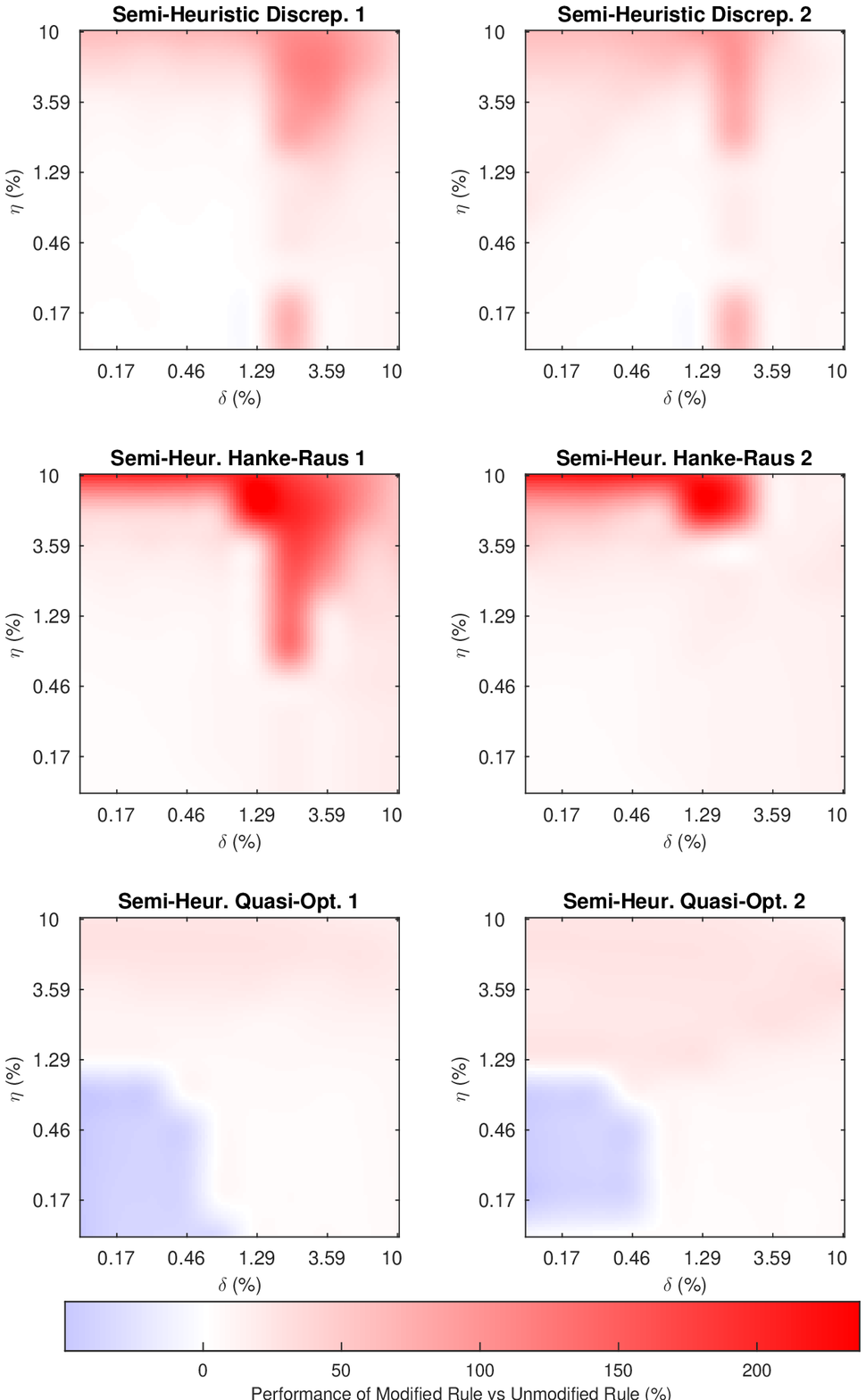}
	\caption{Fredholm operator of the first kind perturbed  by heat operator: set-up identical to Figure~\ref{heatdot1}. Red indicates that the semi-heuristic rules perform better than their standard heuristic counterparts and vice versa.}
	\label{heat1}
\end{figure}

In Figure~\ref{heat1}, the plots for the heuristic discrepancy and Hanke-Raus rules demonstrate that the semi-heuristic rules offer an overall improvement for all ranges of operator and data noise. However, we observe that the semi-heuristic quasioptimality rules performs slightly worse for small data and operator noise, but exhibit much better performance when both the mentioned noises are larger. Additionally, one may also observe that the semi-heuristic Hanke-Raus rules perform significantly better than their standard heuristic counterparts for very large operator noise.

\paragraph{Blur operator perturbed by tomography operator}

In a next experiment, we again simulate a deterministic operator perturbation by considering the \texttt{blur} operator from Hansen's tools and perturbing it by the tomography operator from before. For the blur operator, we set $\texttt{band}=8$ and $\texttt{sigma}=0.9$, which is modelled by the Gau{\ss}ian point spread function:
\[
h(x,y)=\frac{1}{2\pi\sigma^2}\exp\left(-\frac{x^2+y^2}{2\sigma^2}\right).
\]
\begin{figure}
	\centering
	\includegraphics[scale=.8]{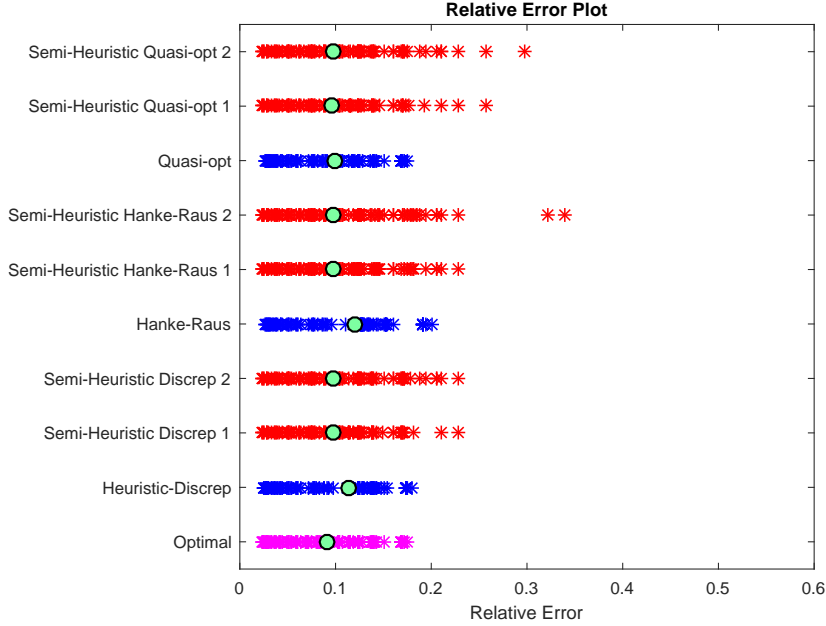}
	\caption{Blur operator perturbed by tomography operator: $D=500$ for SH1, $D=0.2$ for SH2 and $\gamma=0.01\times\eta$.}
	\label{blurdot1}
\end{figure}

In Figure~\ref{blurdot1}, we observe as before that the semi-heuristic rules exhibit improvements over their standard counterparts for the heuristic discrepancy and Hanke-Raus rules, although the standard quasioptimality rule performs quite well and in this case, its semi-heuristic variants do not necessarily present a better choice.

\begin{figure}
	\centering
	\includegraphics[scale=.8]{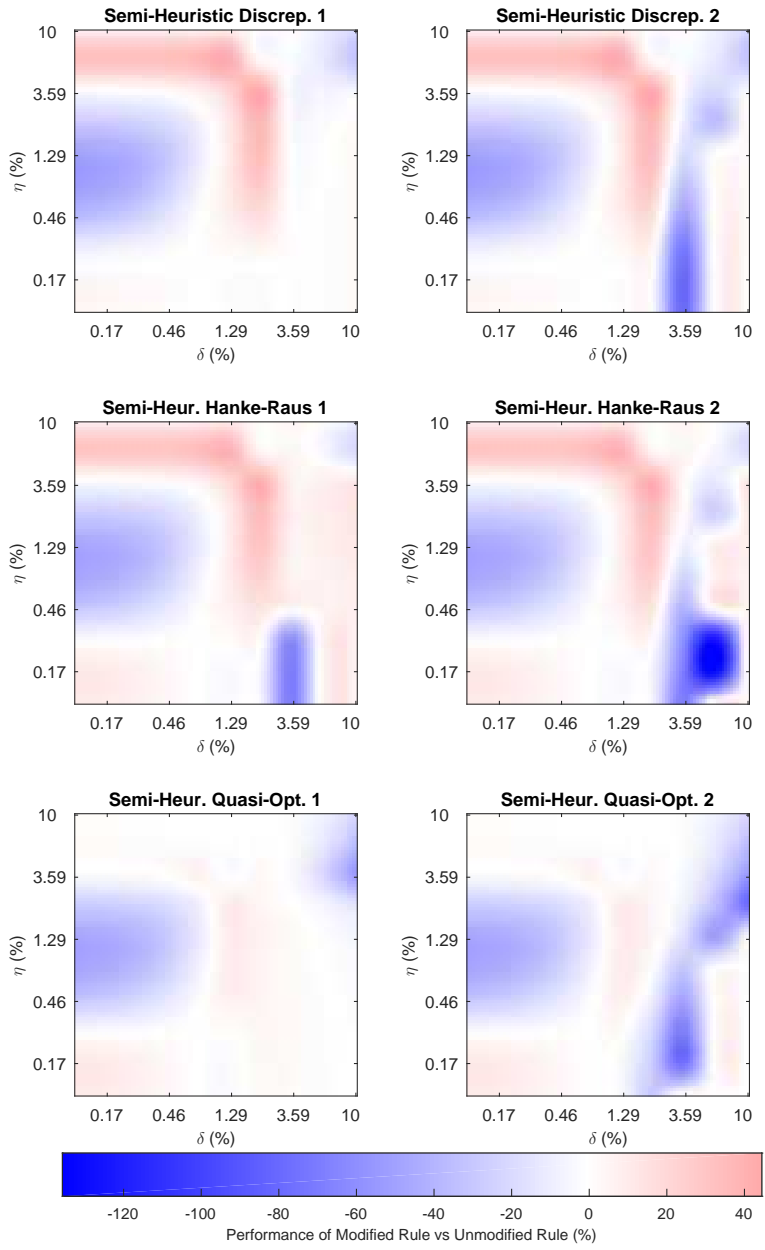}
	\caption{Blur operator perturbed by tomography operator: set-up identical to Figure~\ref{blurdot1}. Red indicates that the semi-heuristic rules perform better than their standard heuristic counterparts and vice versa.}
	\label{blur1}
\end{figure}

In Figure~\ref{blur1}, it is difficult to draw any meaningful conclusions, although it seems that for large operator noise and reasonable data noise, the semi-heuristic discrepancy and Hanke-Raus rules perform better than the standard heuristic rules. Consequently, one may conclude that for many situations, the semi-heuristic rules offer an improvement on their standard heuristic counterparts.

Note that in all experiments,  the minimiser in the range $[\lambda_{\text{min}},\alpha_{\text{max}}]$ of the standard heuristic functionals was occasionally $\alpha_{\text{max}}$; particularly when the operator noise was large. Note that we rectified this failure by the interior minima search as described above. Had we not rectified this failure, the improvement of the semi-heuristic methods would have been even greater pronounced.

\section{Conclusion}
In this paper, we presented a modification of the standard heuristic parameter choice 
rules in the case of a known bound on 
the operator
perturbation but unknown data noise-level. 
In particular, the modifications were two-fold: the introduction of a compensating function and an
appropriately selected lower bound, the motivations for which have been covered.

We proved convergence of the modified rules as
the data and operator errors tend to zero provided that the noise condition holds and the lower bound of the regularisation parameter
satisfies certain condition.


The numerical experiments confirmed that the semi-heuristic methods may yield an improvement over the standard parameter choice rules in many situations. Incidentally, the optimal choices of $D$ and $\gamma$ presents room for further research.

\section*{Acknowledgements}
This work was supported by the Austrian Science
Fund (FWF) project P 30157-N31. The research of U. H\"amarik and U. Kangro was supported by institutional research funding IUT20-57 of the Estonian Ministry of Education and Research.

\bibliographystyle{siam}
\bibliography{lib}

\end{document}